\documentclass[12pt,reqno]{amsart}
\usepackage{latexsym, amsmath, amscd, amssymb, amsthm, bm, mathrsfs}
\usepackage[colorlinks=true,linkcolor=blue,urlcolor=blue,
  citecolor=blue]{hyperref}
\usepackage{amsrefs}
\usepackage[T1]{fontenc}
\usepackage{mathptmx}
\usepackage{microtype}
\usepackage{graphicx} 
\usepackage{pgf}

\usepackage[centering, includeheadfoot, hmargin=1.2in, vmargin=0.8in,
  headheight=30.4pt]{geometry}
  \usepackage{xypic}
\newtheorem{lemma}{Lemma}[section]
\newtheorem{thm}[lemma]{Theorem}
\newtheorem{cor}[lemma]{Corollary}
\newtheorem{prop}[lemma]{Proposition}

\theoremstyle{definition}
\newtheorem{definition}[lemma]{Definition}
\newtheorem{remark}[lemma]{Remark}
\newtheorem{example}[lemma]{Example}
\theoremstyle{remark}
\newcommand{\mb}{\mathbf}

\newcommand{\pr}[1]{\mathbb{P}^{#1}}
\newcommand{\ol}{\mathcal{O}}
\newcommand{\V}{\mathcal{V}}
\newcommand{\E}{\mathcal{E}}
\newcommand{\F}{\mathcal{F}}
\newcommand{\K}{\mathbb{K}}
\newcommand{\A}{\mathcal{A}}
\newcommand{\Z}{\mathbb{Z}}
\newcommand{\Q}{\mathbb{Q}}
\newcommand{\R}{\mathbb{R}}
\newcommand{\C}{\mathbb{C}}
\newcommand{\N}{\mathbb{N}}
\newcommand{\T}{\mathbb{T}}

\newcommand{\eL}{\mathcal{L}}

\newcommand{\trop}{\operatorname{trop}}

\newcommand{\im}{\operatorname{Im}}

\newcommand{\codim}{\operatorname{codim}}

\newcommand{\Conv}{\operatorname{Conv}}

\newcommand{\Vol}{\operatorname{Vol}}

\newcommand{\into}{\hookrightarrow}

\newcommand{\df}{\operatorname{def}}

\renewcommand{\geq}{\geqslant}
\renewcommand{\leq}{\leqslant}

\begin{document}

\title[Higher duality and toric embeddings]{Higher order duality and toric embeddings}

\author[A.~Dickenstein]{Alicia Dickenstein}
\address{Alicia Dickenstein \\ Department of Mathematics, FCEN \\
Universidad de Buenos Aires \\ and IMAS - CONICET \\ Ciudad Universitaria - Pab. I
\\ C1428EGA Buenos Aires \\Argentina}
\email{\href{mailto:alidick@dm.uba.ar}{alidick@dm.uba.ar}}
\urladdr{\href{http://mate.dm.uba.ar/~alidick}{mate.dm.uba.ar/~alidick}}
\thanks{AD is partially supported by UBACYT 20020100100242, CONICET PIP 112-200801-00483,
  and ANPCyT PICT 2008-0902, Argentina. SDR is partially supported by VR grant NT:2010-5563}

\author[S.~Di~Rocco]{Sandra Di Rocco}
\address{Sandra Di Rocco\\ Department of Mathematics\\ Royal Institute of
  Technology (KTH)\\ 10044 Stockholm\\ Sweden}
\email{\href{mailto:dirocco@math.kth.se}{dirocco@math.kth.se}}
\urladdr{\href{http://www.math.kth.se/~dirocco}{www.math.kth.se/~dirocco}}

\author[R.~Piene]{Ragni Piene}
\address{Ragni Piene\\CMA/Department of Mathematics\\
University of Oslo\\P.O.Box 1053 Blindern\\NO-0316 Oslo\\Norway}
\email{\href{mailto:ragnip@math.uio.no}{ragnip@math.uio.no}}
\urladdr{\href{http://www.mn.uio.no/math/english/people/aca/ragnip/index.html}{www.mn.uio.no/math/english/people/aca/ragnip/index.html}}


%\date{\today}

\begin{abstract}
The notion of higher order dual varieties of a projective variety, introduced by Piene in 1983,  
is a natural generalization of the classical notion of projective duality. In this paper we 
study higher order dual varieties of  projective toric embeddings.   
We express the degree of the second dual variety of a $2$-jet spanned embedding of a smooth 
toric threefold  in geometric and  combinatorial terms, and  we classify those  whose second dual variety has 
dimension less than expected.  We  also describe the tropicalization of all higher order dual varieties of an  
equivariantly embedded (not necessarily normal) toric variety.

\bigskip

\centerline{\em Dedicated to the memory of our friend Mikael Passare (1959--2011)}
\end{abstract}

\maketitle

\section{Introduction} \label{sec:intro}

Projective duality of algebraic varieties is a classical subject in algebraic geometry. 
Given an embedding $X\into \pr{m}$ (over an algebraically closed field of characteristic $0$), 
the Zariski closure of the set of hyperplanes $H\in {\pr{m}}^\vee$  tangent to $X$ at a smooth point  
is an irreducible variety called the {\em dual variety} and denoted by $X^\vee$.  
There is a natural  generalization of projective duality, introduced in \cite{P83}. 
One defines the {\em $k$-th dual variety} $X^{(k)} \subset {\pr{m}}^\vee$ 
of $X\subset\pr{m}$ as the Zariski closure of the set of hyperplanes tangent 
to $X$ to the order $k$ at some smooth point $x\in X$   (see Definition~\ref{def:Xk}.)

The purpose of this paper is to introduce higher order dual varieties for toric 
embeddings and give  different characterizations. Projective duality for toric varieties 
is of particular interest because of the connection with convex geometry and combinatorics. 

Consider a (non-degenerate) embedding of a projective variety $X\into \pr{m}$.  For $k=1$, 
the expected codimension of the dual variety $X^\vee =X^{(1)}$ is $1$. 
When this is not the case, $X$ is said to be (dual) defective. 
Defective embeddings have been studied and classified (see e.g. \cites{E85,E86} and, for the toric case,
\cites{DR06,CDR08}).  
A combinatorial characterization of the dimension, 
as well as a positive formula (that is, a formula involving only positive
terms) for the degree of the dual variety of 
toric embeddings via tropical geometry, were recently given in \cite{DFS07}. 

Likewise, for any $k$, the expected dimension of $X^{(k)}$ equals 
$m+n-\ell$, where $n = \dim(X)$ and $\ell$ is the generic rank of the $k$-jet map (see Definition \ref{kjet}). 
Then $X$ is said to be $k$-defective when $X^{(k)}$ has lower dimension than expected.
Typical examples of varieties that are $k$-defective for $k\ge 2$ are scrolls. 
Higher order dual varieties of scrolls over curves have been studied 
in \cite{PS84} and \cite{MP91}.  Notice that in this case  
the osculating spaces do not  have maximal dimension (see Example \ref{segre}.)  
We shall mainly restrict our  
attention to $k$-jet spanned embeddings, i.e., embeddings with
 the property that  the $k$-jet map has maximal rank $\ell = \binom{n+k}n$ at every smooth point.
 However, in Section \ref{RNC} we shall examine the case of rational normal scrolls, which are typical 
examples of embeddings that are not $k$-jet spanned.
 
 Smooth $k$-jet spanned  embeddings of surfaces  which are $k$-defective
 have been classified in \cite{LM01}. In fact, 
 there is only one defective case, namely $(\pr{2}, 
 {\mathcal O}(k)).$ This result generalizes the classification of 
 dual defective smooth surfaces.  In Corollary~\ref{cor:classification}
 we extend this classification to smooth toric threefolds for $k=2$. 
Furthermore, we give in Theorem~\ref{thm:threefolds} a formula for the degree 
of the second dual variety in terms of 
 combinatorially defined quantities (see also Corollary~\ref{cor:deg2}).
 We summarize these results in the following statement.
 
 Let $(X,\mathcal L)$ be a $2$-jet spanned toric 
 embedding of a smooth threefold, corresponding to a $2$-regular  lattice polytope $P$ of
 dimension three (cf. Definition~\ref{def:kreg} 
 and Proposition~\ref{prop:DR99} for the notion of $2$-regularity).
Then $X$ is $2$-defective if and only if $(X, \mathcal L)=(\mathbb P^3,\mathcal O(2))$, 
in which case the second dual variety $X^{(2)}$ is empty.
Moreover:
\begin{enumerate} 
\item[1)] $\deg X^{(2)} =120$ if $(X, \mathcal L)=(\mathbb P^3,\mathcal O(3))$.
\item[2)] $\deg X^{(2)}=6(8(a+b+c)-7)$ if $(X,\mathcal L)=
(\mathbb P(\mathcal O_{\mathbb P^1}(a),\mathcal O_{\mathbb P^1}(b), 
\mathcal O_{\mathbb P^1}(c)), 2\xi)$, where $a,b,c \ge 1$ and $\xi$ 
denotes the tautological line bundle.
\item[3)] In all other cases,
\[ \deg_2 \, X^{(2)}= 62 \Vol(P) -57\F+28\E-8\V+58 \Vol(P^\circ) +51\F_1+20\E_1\]
where $\Vol(P)$, $\F$, $\E$ (resp. $\Vol(P^\circ)$, $\F_1$, $\E_1$) denote the (lattice) 
volume, area of facets, length of edges of $P$ 
(resp. of the convex hull of the interior lattice points of $P$), and $\V$ is the number 
of vertices of $P$. (The definition of $ \deg_2 \, X^{(2)}$
is given in~\eqref{eq:degk}, cf. also Proposition~\ref{kdual}.)
\end{enumerate}

A (not necessarily normal) 
non-degenerate equivariantly embedded projective toric variety is 
rationally parameterized by monomials with exponents
a lattice point configuration 
 $\A = \{r_0, \dots, r_m\} \subset \Z^n$ (see \cite{GKZ}). That is, $\A$ is a
subset of the lattice points in the  lattice polytope $P =\Conv (\A)$. 
We denote this variety by $X_\A$.  
In Section \ref{def} we give  the description of the tropicalization $\trop(X_\A^{(k)})$
 of the $k$-th dual variety of $X_\A$ for any $k$. 
 By a result of Bieri and Groves \cite{BG84}, $\trop(X_\A^{(k)})$ is a rational 
 polyhedral fan of the same dimension 
as \[X_\A^{(k)}\cap \{x \in  {\pr{m}}^\vee, \, x_i \not=0 \text{ for all } 
i = 0, \dots, m\}.\]
 In Theorems~\ref{th:trop} and~\ref{thm:euler} 
we present characterizations of  $\trop(X_\A^{(k)})$, which are direct generalizations
of the corresponding results for the classical dual variety obtained in
\cites{DFS07,DT11}. This leads to a combinatorial characterization
of the dimension and to a positive formula for the degree of higher dual varieties 
(cf. Remark~\ref{rmk:extension}).

\section{Higher order dual varieties}
Let $\iota\colon X\into \pr{m}$ be an 
embedding of a complex non-degenerate algebraic
variety of dimension $n$. The 
{\em $k$-th dual variety} $X^{(k)} \subset {\pr{m}}^\vee$ 
of $\iota$ is the Zariski closure of the set of hyperplanes tangent 
to $X$ to the order $k$ at some smooth point $x\in X$. 
Let us make the concept of tangency to a certain order more precise. 
 Recall that a hyperplane $H$ is tangent to $X$ at a smooth point $x$ 
if and only if $H$ contains the embedded 
 tangent space $\T_{X,x}$. 
Let  $(x_1,...,x_n)$ be a local system of coordinates around $x$, 
so that the maximal ideal is ${\mathfrak m}_x=(x_1,...,x_n)$ 
in ${\mathcal O}_{X,x}$, and let $\mathcal
L:=\iota^*({\mathcal O}_{\pr{m}}(1)).$ 

 The vector space $\mathcal L/{\mathfrak m}_{x}^{k+1}\mathcal L$
is the  fibre at $x\in X$ of the $k$-th principal parts (or jet)
sheaf ${\mathcal P}_{X}^k(\mathcal L)$, which has generic rank
$\binom{n+k}{n}$. We identify $H^0(\pr{m},{\mathcal O}(1)) \otimes {\mathcal O}_X \simeq {\mathcal O}_X^{m+1}$.
The  $k$-th jet map (of coherent sheaves) 
$$j_k: {\mathcal O}_X^{m+1}\to  {\mathcal P}_{X}^k(\mathcal L)$$
is given fiberwise by the linear map  $
j_{k,x} : H^0(\pr{m}, \mathcal O(1))\to H^0(X, \mathcal L) \to 
H^0(X,  \mathcal L /{\mathfrak m}_x^{k+1}\mathcal L)$, 
induced by the map of ${\mathcal O}_{X}$-modules
$\mathcal L\to \mathcal L/{\mathfrak m}_x^{k+1}\mathcal L$. 
So if $s\in H^0(X, \mathcal L)$, $j_{k,x}(s)$ is the Taylor series expansion up to order $k$ of $s$ with respect to the local coordinates $x_1,\ldots,x_n$. With respect to the natural basis $\{1,dx_1,\ldots, dx_n,\ldots,dx_n^k\}$ for ${\mathcal P}_{X}^k(\mathcal L)_x$, it can be written as
%In local coordinates it is simply defined as  
\[j_{k,x}(s)=(s(x),\frac{
\partial s}{
\partial x_1}(x),\ldots, \frac{
\partial s}{
\partial x_n}(x),\frac{1}{2}\frac{
\partial^2 s}{
\partial x_1^2}(x),\ldots,\frac{1}{2} \frac{
\partial^2 s}{
\partial x_n^2}(x),\ldots).\]

Thus, $\pr{}(\im(j_{1,x}))=\pr{}({\mathcal P}_{X}^1(\mathcal L)_x)={\T}_{X,x}\cong \pr{n}$ 
is the embedded tangent space at the point $x$. More generally,
the linear space $\pr{}(\im(j_{k,x}))={\T}^{k}_{X,x}$ is called the
{\it $k$-th osculating space} at $x.$ 

\begin{definition} \label{def:Xk}
We say that a hyperplane $H$  is tangent to $X$ to order $k$  at a 
smooth point $x$ if ${\mathbb T^k_{X,x}}\subseteq H$.
The $k$-th dual variety is
\begin{equation}\label{eq:Xk}
X^{(k)}:=\overline{\{H\in {\pr{m}}^\vee\; |\, H\supseteq {\mathbb T^k_{X,x}}
\text{ for some }\, x\in X_{\rm sm}\}}.
\end{equation}
\end{definition}

In particular, $X^{(1)}=X^\vee$. Alternatively, one can define $X^{(k)}$ 
as the closure of the image of the map
\begin{equation}\label{mapkdual}
\gamma_k :  \mathbb P(({\rm Ker}\, j_k)^\vee|_{X_{k{\rm -cst}}})
\to {\mathbb P^m}^\vee,\end{equation}
where $X_{k{\rm -cst}}$ denotes the open set of $X$ where 
the rank of $j_k$ is constant. 
Note that $X^{(k)}\subseteq X^{(k-1)}$.  Moreover,
$X^{(2)}$ is contained in the  singular locus of $X^\vee$, 
since a necessary condition for a point $H\in X^\vee$ to be smooth, 
is that the intersection $H\cap X$ has a singular point of multiplicity 2: 
if $H\supseteq{ \mathbb T^k_{X,x}}$ for $k\ge 2$, then $H\cap X$ has a singular
point of multiplicity $\ge k+1$.

%Because the concept of tangency to the order $k$ should imply the vanishing of 
%all derivatives up to the order $k$, we have to
%require $\T_{X,x}^{(k)}\cong \pr{\binom{n+k}k-1}$. 
%Of course this does not always happen. 

\begin{definition}\label{kjet} 
We say that the embedding $\iota\colon X\into\pr{m}$ is {\em $k$-jet spanned} 
at a smooth point $x\in X$
if the $k$-th osculating space to $X$ at $x$ has
the maximal dimension, $\binom{n+k}k-1$, or, 
 equivalently, the map $j_{k,x}$ is surjective.
We say that $\iota$ is $k$-jet spanned if it 
is $k$-jet spanned at  all smooth points $x\in X$.
\end{definition}

\begin{example}\label{segre}
Consider the Segre embedding $\pr{t}\times \pr{s}\into \pr{(t+1)(s+1)-1}$. 
Any hyperplane section  is given,  locally around a smooth point with local coordinates  
$(x_1,\ldots x_t, y_1,\ldots ,y_s)$,
by the vanishing of a polynomial of the form 
$c_{0}+\sum_{i=1}^t c_{i0}x_i+\sum_{j=0}^s c_{0j}y_j+$ $\sum_{i=1,\ldots t,j=1,\ldots s} 
c_{ij}x_iy_j.$ One sees that $\dim {\mathbb T^2_{X,x}} =(t+1)(s+1)-1< \binom{s+t+2}2-1$.
In fact, the $2$-jets 
corresponding to $\frac{\partial^{2}}{\partial x_i^{2}}$ and $\frac{\partial^{2}}{\partial y_j^{2}}$
are not ``generated by the embedding.''
\end{example}

%As seen in  Example~\ref{segre},  embeddings of scrolls are not $2$-jet spanned.
 
 If the embedding $\iota$ is (generically) $k$-jet spanned, 
 the (generic) rank of $j_k$ is maximal, namely $\binom{n+k}k$.   When
 $\iota$ is generically $k$-jet spanned and  the map
 $ \gamma_k\colon\mathbb P({\rm Ker}\, j_k^\vee)
 \to {\mathbb P^m}^\vee$ is generically finite, the dimension of $X^{(k)}$ 
 is equal to $m+n-\binom{n+k}k$.
% \begin{definition} \label{def:gammakfinite}
%The \emph{expected dimension} of $X^{(k)}$ is the number $m+n-\binom{n+k}n$ . 
When the general fibers of $\gamma_k$ have positive dimension, the $k$-dual variety has 
lower dimension than expected, and in this case we say that the embedding is \emph{$k$-defective}, 
with positive \emph{$k$-dual defect}
\[\df_k (X):=m+n-\binom{n+k}n-\dim X^{(k)}.\]
When $X^{(k)}$ has the expected dimension, we set
\begin{equation} \label{eq:degk}
\deg_k X^{(k)} := \deg(\gamma_k)\deg X^{(k)}.
 \end{equation}
%\end{definition}
\medskip

%One can define jets supported at a finite number of points, which leads to 
%the more general notion of $k$-jet ampleness. 

The notion of $k$-jet spannedness at a point can be generalized in the following way.
 
 \begin{definition} \label{def:kjetample}
 An embedding $\iota\colon X\into \pr{m}$, with $\eL=\iota^*\mathcal O(1)$, 
 is said to be \emph{$k$-jet ample} 
 if for every collection of points $x_1,\ldots,x_t\in X$ and integers 
 $k_1,\ldots,k_t$ such that $\sum k_i=k+1$, the map $H^0(\pr{m}, \mathcal O(1)) \to H^0(X,\eL)
 \to \oplus _{i=1}^t H^0(X, \eL /m_{x_i}^{k_i}\eL)$ is surjective.
 \end{definition}
 
 Note that $1$-jet ampleness is the same as very ampleness.  For more details and characterizations of $k$-jet ampleness for several 
 classes of embeddings we refer the reader to \cites{DR99,BDRS1}.

The following is essentially proven in Theorem~1.4 and Proposition~2.4 in \cite{LM99}. 
We give a proof for completeness, and in order to include (c). As usual, $c_i(\mathcal E)$ 
denotes the $i$th Chern class of a vector bundle $\mathcal E$.

 \begin{prop}\label{kdual} 
Assume $X$ is a smooth variety of
dimension $n$, and  that the embedding $\iota\colon X\into \pr{m}$ is $k$-jet spanned. Then
 \begin{enumerate}  
 \item  \label{item:a} the embedding $\iota$ is $k$-defective 
 if and only if $c_n(\mathcal P_X^k(\mathcal L))=0$;
 \item if $\iota$ is not $k$-defective, then 
 $\deg_k X^{(k)}=c_n(\mathcal P^k_X(\mathcal L))$;
 \item if $\iota$ is generically $(k+1)$-jet spanned, then the embedding is not $k$-defective;
 \item  if $\iota$ is $(k+1)$-jet ample, then $\deg(\gamma_k)=1$, and thus $\deg_k X^{(k)} =
 \deg X^{(k)}$.
\end{enumerate}
 \end{prop} 
   
  \begin{proof}
   Let $\mathcal K_k$ denote the kernel of the $k$th jet map 
   $j_k\colon \mathcal O_X^{m+1}\to \mathcal P_X^k(\mathcal L)$. 
   Because $X$ is smooth, $ \mathcal P_X^k(\mathcal L)$ is a 
   locally free sheaf, with rank $\binom{n+k}k$. Then,  
   since $\iota$ is $k$-jet spanned and thus $j_k$ is 
   surjective, the sheaf  $\mathcal K_k$ is also locally free, 
   with rank $m+1-\binom{n+k}k$. So the sequence
  \[ 0\to \mathcal K_k \to \mathcal O_X^{m+1}\to \mathcal P_X^k(\mathcal L)\to 0\]
  is exact. Set $\pi\colon \mathbb P(\mathcal K_k^\vee)\to X$,
  and consider the composed map
  \[ \gamma_k\colon \mathbb P(\mathcal K_k^\vee)\subset X\times {\mathbb P^m}^\vee \to {\mathbb P^m}^\vee.\]
 Then $ X^{(k)}$ is equal to the image of $\gamma_k$ and 
 $\mathcal O_{\mathbb P(\mathcal K_k^\vee)}(1)=
\gamma_k^*\mathcal O_{{\mathbb P^m}^\vee}(1)$. From the exact sequence above we have
\[c(\mathcal P^k_X(\mathcal L))=c(\mathcal K_k)^{-1},\]
and since $c(\mathcal K_k)^{-1}=s(\mathcal K_k^\vee)$, where $s$ denotes the Segre class,
we have
   \[c_n(\mathcal P^k_X(\mathcal L))=  s_n(\mathcal K_k^\vee)=
   \pi_*c_1(\mathcal O_{\mathbb P(\mathcal K_k^\vee)}(1))^{n+m-\binom{n+k}k},\]
   where  the last equality is a well known expression for the $n$th Segre class.

(a) and (b):  Assume  $\dim X^{(k)}=m+n-\binom{n+k}k-\df_k (X),$ with $\df_k (X)>0$.  
 Then, for dimension reasons, 
  $c_1(\mathcal O_{{\mathbb P^m}^\vee}(1))^{n+m-\binom{n+k}k}|_{X^{(k)}}=0$, therefore 
$c_1(\mathcal O_{\mathbb P(\mathcal K_k^\vee)}(1))^{n+m-\binom{n+k}k}=0$, 
and hence also $c_n(\mathcal P^k_X(\mathcal L))=0$.
 
Assum instead that $\dim X^{(k)}=m+n-\binom{n+k}k$. Then $\gamma_k$ is generically finite, and
\[  \deg \gamma_k \cdot \deg X^{(k)}= c_1(\mathcal O_{\mathbb P(\mathcal K_k^\vee)}(1))^{n+m-\binom{n+k}k}=
\pi_* c_1(\mathcal O_{\mathbb P(\mathcal K_k^\vee)}(1))^{n+m-\binom{n+k}k}=c_n(\mathcal P^k_X(\mathcal L)).\]

(c):   Assume now that the embedding $\iota$ is $(k+1)$-jet spanned, i.e., 
that the map $H^0(X,\eL) \to H^0(X,\eL/ m_x^{k+2}\eL)$ is surjective for general $x\in X$. 
This implies that the section of a general hyperplane tangent to order $k$ 
at such a point $x$, is, locally around $x$, of the form 
\[s= \sum_{\sum t_i=k+1} a_{t_1,\ldots, t_n}\,\Pi_1^n\,  x_i^{t_i}+\text{ higher degree terms }. \]
In other words, it can be identified with a general element of the 
complete linear system $|{\mathcal O}_{\pr{n-1}}(k+1)|$. 
If the map $\gamma_k$ had  positive dimensional general fibers, 
then the generic element $s$ should have a singularity  
at a point $y$ near  $x$, with $y\neq x$. By Bertini's theorem, 
this point should lie in the base locus of the complete 
system $|{\mathcal O}_{\pr{n-1}}(k+1)|$, which is base point free.

(d):  If the embedding is $(k+1)$-jet ample, then it is $(k+1)$-jet 
spanned, and thus, by (c), not $k$-defective, i.e.,  
the map $\gamma_k$ has zero-dimensional general  fibers.  As before, 
assume that a  hyperplane $H$ is tangent to $X$ at $x$ 
to order $k$, then $H:=(s=0),$ for some $s\in H^0( \eL\otimes m_{x}^{k+1})$.  
By Bertini's theorem, any other singular point $y$ with 
$\gamma_k(x) = \gamma_k(y)$ would have to lie in the base locus of the system 
$|\eL\otimes m_{x}^{k+1}|.$ But the $(k+1)$-jet ampleness implies that 
\[H^0(X,\eL\otimes m_x^{k+1}) \to H^0(X, \eL\otimes m_x^{k+1}/m_y\eL)\]
 is surjective for every $y\neq x$,  and thus  the system is base point free. 
 Therefore $\deg(\gamma_k)=1$.
  \end{proof}
  
  We shall denote by $\Omega_X$ the sheaf of K\"abler differentials on a variety $X$. 
  If $X$ is smooth, we let $T_X:=\Omega_X^\vee$ denote the tangent bundle of $X$, and 
 $c_i:=c_i(T_X)=(-1)^ic_i(\Omega_X)$ the $i$th Chern class of $X$. 
 Note that $c_1=-K_X$, where $K_X=c_1(\Omega_X)$ is the class of the canonical divisor.  
 For $\eL$ a line bundle on $X$, we set $L:=c_1(\eL)$.
 \medskip
 
 Let $(X, \mathcal L)$ be a $2$-jet spanned embedded smooth surface. 
 It is shown in page 4829 of \cite{LM99} that 
 $ c_2(\mathcal P_X^2(\mathcal L)) =$ $  \frac{5}3\bigl((3L-2c_1)^2+3c_2-c_1^2\bigr)$.
 For a smooth threefold, we get the following result.
 
 \begin{prop}\label{prop:degree}
 Assume $(X,\mathcal L)$ is a $2$-jet spanned embedding of a smooth threefold. 
Then, $\deg X^\vee =4L^3-3c_1L^2+2c_2L-c_3$ and 
\[
c_3(\mathcal P_X^2(\mathcal L)) =120L^3-180c_1L^2+48c_2L+72c_1^2L-7c_1^3-20c_1c_2-8c_3.
\]
\end{prop}

\begin{proof}
The  exact sequences
 \begin{equation}\label{jetsequence} 0\to S^i\Omega_X\otimes \mathcal L\to {\mathcal P}_{X}^i(\mathcal L)
\to {\mathcal P}_{X}^{i-1}(\mathcal L)\to 0\end{equation}
 allow one to recursively compute the Chern classes of the twisted 
 jet bundles in terms of those of $X$ and $\mathcal L$. 
 We get
  \begin{align*}
 c_1(\mathcal P_X^1(\mathcal L))=&4L-c_1\\
 c_2(\mathcal P_X^1(\mathcal L))=&c_2(\Omega_X \otimes \mathcal L)+c_1(\Omega_X \otimes \mathcal L)L = 6L^2-3c_1L+c_2\\
 c_3(\mathcal P_X^1(\mathcal L))=&c_3(\Omega_X \otimes \mathcal L)+c_2(\Omega_X \otimes \mathcal L)L =4L^3-3c_1L^2+2c_2L-c_3.
  \end{align*}
It follows that
 \begin{align*}
 c_1(\mathcal P_X^2(\mathcal L))=&c_1(S^2\Omega_X \otimes \mathcal L)+c_1(\mathcal P_X^1(\mathcal L)) = 10L-5c_1\\
 c_2(\mathcal P_X^2(\mathcal L))=&c_2(S^2\Omega_X \otimes \mathcal L)+c_2(\mathcal P_X^1(\mathcal L))+c_1(S^2\Omega_X \otimes \mathcal L)c_1(\mathcal P_X^1(\mathcal L))\\
 =&45L^2-45c_1L+6c_2+9c_1^2\\
c_3(\mathcal P_X^2(\mathcal L))=&c_3(S^2\Omega_X\otimes \mathcal L)+c_3(\mathcal P_X^1(\mathcal L))\\
 &+c_2(S^2\Omega_X\otimes \mathcal L)c_1(\mathcal P_X^1(\mathcal L))+c_1(S^2\Omega_X\otimes \mathcal L)c_2(\mathcal P_X^1(\mathcal L))\\
 =&120L^3-180c_1L^2+48c_2L+72c_1^2L-7c_1^3-20c_1c_2-8c_3.
 \end{align*}
 \end{proof}

It is a classical result for projective varieties (in characteristic $0$)
that a general tangent hyperplane is tangent along a linear subspace of the variety, 
or in other words, a general \emph{contact locus} is a linear space. These contact 
loci are naturally identified with the fibers of $\gamma_1.$ If the embedding is 
not defective, then the general fibers of $\gamma_1$ are $0$-dimensional, hence  
equal to one point, which implies that $\gamma_1$ is birational.
For higher order tangency, it is not clear what to expect of the higher contact loci, 
i.e., of the general fibers of $\gamma_k$.
One might intuitively search for $k$-defective embeddings among 
varieties containing linear spaces embedded as  higher order Veronese varieties. 
The following example shows that this intuition is in fact wrong.

\begin{example}\label{abc}

Let $a,b,c \ge 1$ and 
$\mathcal O(a), \mathcal O(b), \mathcal O(c)$, the corresponding
line bundles on ${\mathbb P}^1$. Consider
$X=\mathbb P(\mathcal O(a)
\oplus \mathcal O(b)\oplus \mathcal O(c))$, with the embedding given 
by the line bundle $\mathcal L=2\xi$, where $\xi$ is the tautological bundle. 
Observe that the embedding given by  $2\xi$ is $2$-jet spanned.  
Let $F$ denote the class of a fiber of the projection $X\to \mathbb P^1$, $\ell=c_1(\xi)$,  and 
$c_i = (-1)^ic_i(\Omega_X), i=1,2,3$. Set $x=a+b+c$.
As $X$ is a  $\mathbb P^2$-bundle  over $\mathbb P^1$, we have
$c_3= 6, c_1c_2= 24, c_1^3=54$. Then, $ L= 2\ell$, $\ell^3=x$, $F\ell^2=1$.
Moreover,
 $c_1=3\ell-(x-2)F$, $c_2= 3\ell^2-2(x-3)F\ell$, $c_1^2=9\ell^2-6(x-2)F\ell$,
$ c_1^2L=6(x+4)$, $c_2L=2(x+6)$, and  $c_1L^2=8(x+1)$.
 The degree  $\deg X^\vee=c_3(\mathcal P_X^1(\mathcal L))$ of the dual variety is equal to
 \begin{equation*}
4L^3 -3c_1L^2+2c_2L-c_3
 =4\cdot 8x- 3\cdot 8(x+1)+ 2\cdot 2(x+6) -6= 6(2(a+b+c)-1).
 \end{equation*}
The second dual variety is also non defective since
\begin{align*}
c_3(\mathcal P_X^2(\mathcal L)) = & 120L^3-180c_1L^2+48c_2L+72c_1^2L-7c_1^3-20c_1c_2-8c_3\\
=&6(8(a+b+c)-7).
 \end{align*}
 \end{example}

  \section{Second dual varieties of smooth toric threefolds} \label{sec:stt}

Our aim in this section is to characterize the degree of the second dual varieties of
$2$-spanned smooth \emph{toric} embeddings in combinatorial terms. Before presenting
our main results (Theorem~\ref{thm:threefolds} and Corollary~\ref{cor:classification}), 
we need to recall some
background notions and previous results.

\subsection{Preliminaries}

 Recall that a toric variety of dimension $n$  is a (not necessarily normal) 
 algebraic variety  containing an algebraic torus $(\C^*)^n$ as Zariski open set 
 and such that the multiplicative self-action of the torus extends to the whole variety $X.$
  A lattice point configuration
(i.e., a finite subset) 
$\A=\{{r}_0,\ldots,{r}_m\}\subset\Z^{n}$ defines a map $\iota_\A:(\C^*)^{n}\to\pr{m},$ by
sending ${\mathbf x}=(x_1,...,x_n)\mapsto ({\mathbf x}^{{r}_0}:\cdots: {\mathbf x}^{{r}_m})$, where
${\mathbf x}^{{\bf r}_i}=\Pi x_j^{r_i^j}$ if ${r}_i=(r_i^1,\ldots,r_i^n)$. 
The Zariski-closure of $\im(\iota_{\A})$ is 
a projective toric variety which we denote by $X_\A$. 
The associated ample line bundle will be denoted by $\eL_\A$.  
Note that the dimension of $X_\A$ equals the dimension of the affine span of $\A$.
We will  assume without loss of generality that $\Z \A = \Z^n$.

\begin{example} \label{ex:q=2}
 Let $d \in \N$ and 
 consider the following matrices
$$
\begin{array}{cc}
{ B}=\left(
\begin{array}{cccccc} 1&1&1&1& \cdots &1\\
0&1&2&3& \cdots & d\\
0&0& 1&3& \cdots &\frac{d(d-1)}{2}
\end{array}\right), \, & \,
{ C}=\left(
\begin{array}{cccccc} 1&1&1&1& \cdots &1\\
0&1&2&3& \cdots & d\\
0& 1& 4&9& \cdots & d^2
\end{array}\right).
\end{array}$$
Then, ${ C} \, = \, M B$, where $M \in {\rm {GL}(3,\Q})$
is the matrix
\[
M=\left(
\begin{array}{ccc} 
1&0&0\\
0&1&0\\
0&1& 2
\end{array}\right).
\]
Let ${\mathcal B}, {\mathcal C}$ be the respective configurations of column vectors.
Then $X_{\mathcal B} = X_{{\mathcal C}}$, with the only difference that
$\iota_{{\mathcal B}}$ is $1$ to $1$ while $\iota_{\mathcal C}$ is $2$ to $1$.
\end{example}
 
We will call an equivariant embedding of a toric variety a \emph{toric embedding}. 
Non-degenerate toric embeddings are in one to one correspondence with lattice
configurations up to affine equivalence. 

In this section we will assume that
$\A \subset \Z^n$ is a point configuration such that
$P=\Conv(\A)$ is a smooth lattice polytope of dimension $n$ and $\A = P \cap \Z^n$. 
Such a polytope $P$ gives a \emph{smooth} toric embedding.
We shall be concerned with $k$-jet spanned smooth toric embeddings, 
which have a simple characterization in terms of the corresponding polytope. 

\begin{definition}\label{def:kreg}
A lattice polytope is called $k$-regular if  all its edges have length at least $k$.
\end{definition}

The following proposition was proved in \cite[Thm. 4.2, p.179; Prop. 4.5, p. 181]{DR99}:

\begin{prop}\label{prop:DR99}
Let $P$ be a smooth lattice polytope in $\R^n$
of dimension $n$ and set $\A = P \cap \Z^n$. 
 The following statements are equivalent:
\begin{enumerate}
\item The polytope  $P= \Conv(\A)$ is $k$-regular.
\item The toric embedding defined by $\A$ is $k$-jet spanned.
\item The toric embedding defined by $\A$ is $k$-jet ample.
\end{enumerate}
\end{prop}

Observe for example that the Segre embedding in Example \ref{segre} is a smooth 
toric embedding,
associated to the configuration of all the lattice points in the polytope 
$\Delta_s\times \Delta_t$, where $\Delta_m$ denotes the 
unimodular simplex of dimension $m$. In this case, all edges have length $1$. 
The embedding is indeed not $2$-jet spanned.

\begin{remark}\label{adjoint-ample}
For a $k$-jet spanned embedding $(X, \eL)$ of dimension $n$, 
the fact that the map $j_k$ is surjective implies that 
${\mathcal P}_{X}^k(\mathcal L)$ is globally generated. 
It follows that $\det({\mathcal P}_{X}^k(\mathcal L))=
\frac{1}{k}\binom{k+n}{n+1}(kK_X+(n+1)\eL)$ is globally generated 
and thus $kK_X+(n+1)\eL$ is nef. (Recall that a line bundle (divisor) 
on a variety $X$ is called \emph{nef} if it intersects every curve 
non-negatively.) For toric manifolds more is true:
a characterization by Musta\c{t}\u{a} (see \cite{M02}) implies that if 
$\eL_\A$ defines a $k$-jet spanned toric smooth  embedding, 
then $kK_X+(n+1)\eL_\A$ is ample, unless $(X,\eL_\A)=(\pr{n},\mathcal O_{\pr{n}}(k))$.
This can be restated as follows: if $P\neq k\Delta_n$ is a smooth
convex lattice polytope of dimension $n$ 
which is  $k$-regular, then the convex hull $\Conv({\rm int}(m P))$ of the interior 
lattice points of the $m$-th dilated polytope $mP$
is a lattice polytope normally equivalent to $P$ (i.e.,
they have the same inner normal fan) for $m\geq\lceil{\frac{n+1}k\rceil }$.
\end{remark}

 As we saw in Proposition \ref{kdual}, for a $k$-jet spanned embedding, 
 the top Chern class  of the $k$-th jet bundle $\mathcal P_X^k(\mathcal L))$ is $0$ if and only if 
 the embedding is $k$-defective. When it is not $0$, it is positive, and is related 
 to the degree of $X^{(k)}$. From the exact sequences  (\ref{jetsequence}) it follows that $c_n(\mathcal P_X^k(\mathcal L))$ can be expressed as a polynomial in the Chern classes $c_i=(-1)^ic_i({\Omega_X})$ and $L=c_1(\mathcal L)$.

 Recall that the Chern classes of a toric variety $X$ corresponding to a 
 polytope $P$ can be expressed in the following way (see e.g. \cite[Cor. 11.8]{D78}):
 \[c_i=\sum_{F\subseteq P, \,\, \codim(F)=i} [F],\]
 where we denote by $[F]$ the class of the invariant subvariety of $X$ 
associated to the face $F$. Toric intersection theory then gives
\[c_i \, L^{n-i} =\sum_{F\subseteq P, \,\, \codim(F)=i} \Vol (F),\]
where $\Vol(F)$ means the lattice volume measured with respect to the lattice 
induced by $\Z^n$ in the linear span of $F$.  For example, for an edge $\xi,$ $\Vol(\xi)=|\xi\cap\Z^n|-1$.

In the case of $k=1$, we saw in Proposition \ref{kdual} that 
\[\deg X^\vee=c_n(\mathcal P_X^1(\mathcal L))=\sum_{i=0}^n (-1)^i (n-i+1) c_i\, L^{n-i},\]
where the last equality can be obtained from the exact sequence (\ref{jetsequence}) with $i=1$.
Thus we recover the following combinatorial expression
 \cite[Thm. 2.8, Ch. 9]{GKZ}, \cite{DR06}:
\[ \deg X^\vee=\sum_{F\subseteq P} (-1)^{\codim(F)} (\dim(F)+1) \Vol(F),\]
where the sum is taken over all faces $F$ of $P$.
 In fact, this formula extends to
 the non-smooth case, with an extra integer factor in each term \cite{MT11}. 
 
 For a lattice polytope $P=\Conv(\A)$ we will use the following notation:
 \[\V=\# \text{ vertices of } P, \,\,  \E=\sum_{\xi \text{ edge of } P} 
 \text{Vol}(\xi),\,\; \F=\sum_{F \text{ facet of } P} \text{Vol}(F ).\]
Then we have
\[L^n=\Vol(P), \,\, c_1L^{n-1}=\F,  \,\, c_{n-1}L=\E, \,\, c_n=\V.\]

\begin{remark}
Higher duals of smooth projective \emph{surfaces} were studied and classified 
by Lanteri and Mallavibarrena
in \cite{LM99}. Using the sequences (\ref{jetsequence}) they computed
\[c_2(\mathcal P_X^k(\eL))=\frac{1}3\binom{k+3}4
\bigl((3L-kc_1)^2+3c_2-c_1^2\bigr),\]
and showed that the right hand side expression is zero only when $(X,\mathcal L)=
(\mathbb P^2,\mathcal O(k))$. Since (see Prop. \ref{kdual}) $\deg_k X^{(k)}=c_2(\mathcal P_X^k(\eL))$ if 
$(X,\eL)$ is $k$-jet spanned, they concluded that $(\mathbb P^2,\mathcal O(k))$ 
is the only such $k$-defective surface.

In the special case of toric surfaces this can be seen as follows. Let $P$ be the 
polytope corresponding to the toric embedding $(X,\eL)$. If 
$P=k\Delta_2$, then $X$ is $k$-defective (since $X^{(k)}=\emptyset$ in this case). So assume 
$P\neq k\Delta_2$. Then (see Remark  \ref{adjoint-ample}) $3L-kc_1$ is very ample, and therefore
$(3L-kc_1)^2> 0$. 
Now $c_2=\V$ is the number of vertices of $P$ and, by Noether's formula 
$c_1^2+c_2=12 \chi(\mathcal O_X)=12$, we get $c_1^2=12-\V$(cf. Corollary~7.4 in~\cite{D78}, based
on previous work of Demazure~\cite{Dem70} for the vanishing of higher cohomologies of the structural
sheaf).  Hence $ 3c_2-c_1^2=4(\V-3)\ge 0$
and we see that $c_2(\mathcal P_X^k(\eL)) > 0$. 

Note that we also get an expression in terms of the polytope:
 \[ \deg_k X^{(k)} = \binom{k+3}4 \bigl(3\Vol(P)-2k\E- \frac{1}3(k^2-4)\V+4(k^2-1)\bigr).\]
\end{remark}
 
 \subsection{New results for threefolds}
 
 In dimension $n\ge 3$, basically nothing is known for higher 
 order dual varieties, except for the case of 
rational normal scrolls, which were studied in \cite{PS84}. 
In the following we consider $2$-jet spanned toric embeddings of threefolds. 
We give a classification of those that are $2$-defective and a formula for the 
degree of the second dual varieties of those that are not.

Adjoint polytopes of toric embeddings $(X_\A, \eL_\A)$ 
provide a powerful classification tool. These are the polytopes 
associated to the line bundles
$r K_\A + j \eL_\A$ for different values of $r,j$. 
In \cite{DHNP11} smooth $3$-dimensional polytopes with no 
interior lattice points are classified using the classification 
of embeddings with high nef-value given by 
Fujita \cite{Fuj87}. (For the definition of nef-value, see e.g. \cite[p. 25]{BS95}.) 
We refer to \cite[Theorem~5.1]{DHNP11} for more details.  
A simple consequence is the following.

\begin{lemma}\label{nef}
Let $\eL_\A$ define a $2$-jet spanned embedding of a smooth toric threefold 
$X_\A$.  If $K_\A+\eL_\A$ is not nef, then the only possibilities are:
\begin{enumerate}
\item $\Conv(\A) = 2\Delta_3$, $3\Delta_3$, i.e., $(X_\A,\eL_\A)=(\pr{3},\ol(a))$, $a=2,3$.
\item $(X_\A,\eL_\A)=(\pr{}(\ol_{\pr{1}}(a)\oplus\ol_{\pr{1}}(b)\oplus\ol_{\pr{1}}(c), 2\xi),$ 
with $a,b,c\geq 2.$
\end{enumerate}
\end{lemma}

\begin{proof}
If $K_\A+\eL_\A$ is not nef, then the nef-value $\tau(\eL_\A)>1=n-2$. 
According to the classification in \cite[Theorem 5.1]{DHNP11},  the only possibilities 
for $P=\Conv(\A)$ are:
\begin{itemize}
\item $P=2\Delta_3$, $2\Delta_3$ as in (a).
\item A Cayley sum of the form $P=P_0\star P_1 \star P_2$, 
where $P_i$ are parallel segments. 
These polytopes contain edges of length $1$ and thus are not 
$2$-regular, by Proposition~\ref{prop:DR99}.
\item There is a smooth polytope $P'$ such that $P'=P\cup\Delta_2$ (i.e., $P$ is a 
blowup polytope).
Such a $P$ is not $2$-regular since it contains three edges of length
$1$.
\item The polytope associated to $K_\A+2\eL_\A$ is  a point.
Then $X_\A$ is not $2$-jet spanned, since in that case $K_\A+2\eL_\A$ would have been ample.
\item There is a projection $P\twoheadrightarrow 2\Delta_2$, 
and the pre-images of the vertices 
are parallel segments. Then $X_\A$ is of the form
$X_\A=\pr{}({\mathcal E})=\pr{}(\ol_{\pr{1}}(a)\oplus \ol_{\pr{1}}(b) \oplus \ol_{\pr{1}}(c))$.
If $P$ is $2$-regular, the segments must have length at least $2$. 
 Notice that  
$\pr{}(\mathcal E) \cong \pr{}(({\ol_\pr{1}}(a)\oplus \ol_{\pr{1}}(b)\oplus 
\ol_{\pr{1}}(c))\otimes \mathcal H),$ for any line bundle $\mathcal H$ on $\pr{1}$. 
Under this isomorphism, $\xi_{\mathcal E\otimes \mathcal H}$ is identified with 
$\xi_\mathcal E\otimes \pi^*(-\mathcal H)$, where $\pi:\pr{}(\mathcal E)\to \pr{1}$ 
is the projection map. 
It follows that we can assume that we are in case {\bf (b)}.
\item  $P=A_0\star A_1$ is a Cayley sum, where $A_i$ 
are smooth polygons. Since $P$ contains edges of length $1$, P is not $2$-regular.
\end{itemize}
\end{proof}

 Recall that  a nef line bundle on a smooth algebraic variety
 is big if and only if its degree is positive. 
For a toric embedding $(X_\A,\eL_\A)$ corresponding to $P=\Conv(\A)$, this means the following.
 Assume $\eL_\A$ is an ample line bundle on the toric variety 
 $X_\A$ such that $K_{X_\A}+\eL_\A$ is nef. Then the
 adjoint polytope $\Conv({\rm int}( P))$ to $P$ (corresponding to $K_{X_\A}+\eL_\A$)
 has positive volume if and only if it is of maximal dimension.
\medskip

 We will use the following notation. Recall that $P=\Conv(\A)$ and $\A=P\cap \Z$. 
Denote by $P^\circ:=\Conv({\rm int}( P) \cap \Z^3)$ the
 convex hull of the set of interior lattice points of $P$, and set
 \[
 \E_1:=  \sum_{\xi \text{ edge of }P^\circ} \Vol (\xi),\,\,\,
 \F_1:=\sum_{F \text{ facet of } P^\circ} \Vol (F).
\]

\begin{thm}\label{thm:threefolds}
Let $(X,\eL):=(X_\A,\eL_\A)$ be a $2$-jet spanned  toric embedding 
of a smooth threefold. Assume $K_{X}+\eL$ is nef.  Then 
\begin{equation}\label{eq:deg2}
\deg_2 \, X^{(2)}= 62\Vol(P)-57\F+28\E-8\V+58\Vol(P^\circ)+51\F_1+20\E_1.
\end{equation}
\end{thm}

\begin{proof} 
We have
$\Vol(P)=L^3$, $\F=c_1L^2$, $\E=c_2L$, and $\V=c_3$. 
Recall that by  Riemann--Roch's theorem, we know that for a threefold $X$, 
 $\chi (\mathcal O_X)=\frac{1}{24}c_1c_2$,
 so that for a toric threefold. Since $X$ is toric, by Demazure vanishing 
 \cite[Prop.~6, p.~564]{Dem70} (see \cite[Cor.~7.4, p.~129]{D78}) we have $\chi (\mathcal O_X)=1$. 
 Hence we get $c_1c_2=24$.
Moreover, $\Vol(P^\circ)=(L-c_1)^3$, $\F_1=c_1(L-c_1)^2$, and $\E_1=c_2(L-c_1)$. 
This allows us to express $c_1^3$ and $c_1^2L$ in terms of volumes:
 \begin{align*}
 c_1^3=&2\Vol(P^\circ)-2\Vol(P)+3\F+3\F_1, \\
 c_1^2L=&\Vol(P^\circ)-\Vol(P)+2\F+\F_1.
  \end{align*}
We then obtain the formula in the statement of the Theorem from the Chern class formula 
for $c_3(\mathcal P^2_X(\mathcal L))$ given in Proposition \ref{prop:degree}. 
\end{proof}

\begin{example}
If $P$ is a cube with edge lengths $2$, then $(X_P,\mathcal L_P)=
(\mathbb P^1\times \mathbb P^1\times \mathbb P^1, 
\mathcal O(2,2,2))$.  Then $\Vol(P)=3!8=48, \, 
\F=6\cdot 2\cdot 4=48,$ $\, \E=12\cdot 2=24, \, \V=8$.
As int$(P)$ is a point, $\Vol (P^\circ)=\F_1=\E_1=0$. Hence 
$\deg_2 X^{(2)}=62\Vol (P)-57\F+28\E-8\V=848$.
\end{example}

As we saw in the proof of Theorem~\ref{thm:threefolds}, 
the Chern numbers of a toric threefold with  $K_{X}+\eL$ nef, 
can be expressed in terms of volumes of $P$  and its first adjoint polytope $P^\circ$.
 In fact, if we take \emph{any} $r$ such that $K_X+r\eL$ is nef, 
 then we have $r\E-\E_r=c_1c_2=24$, where $\E_r$ 
denotes the sum of the edge lengths of the polytope $(rP)^\circ$. 
If $r\ge 1$ is such that $K_{X}+r\eL$  is nef, we can write similar 
formulas in terms of the adjoint polytope 
$(rP)^\circ$. Instead of using  $\F_r$ and $\E_r$, it might 
make more sense to use the volumes $\Vol((rP)^\circ)$, 
since they play a role in  Ehrhart theory. This produces several 
different formulas for the degree of the second
dual variety similar to~\eqref{eq:deg2}.

  \begin{cor}\label{cor:deg2}
 Assume $K_X+3\eL$ is nef. Then
 \begin{align*}
  \deg_2 X^{(2)}=&12\Vol((2P)^\circ)+15\F_2+20\E_2-56\Vol(P)+24\F+8\E-8\V\\
 \deg_2 X^{(2)}=&
 19 \Vol((3P)^\circ)-3 \Vol((2P)^\circ)-126\Vol(P)+54\F+48\E-8\V-480.
  \end{align*}
 \end{cor}
 \begin{proof}
As $K_{X}+3\eL$ is nef,
\begin{align*}
 \Vol((2P)^\circ)=& (2L-c_1)^3=8L^3-12c_1L^2+6c_1^2L-c_1^3, \\
\Vol((3P)^\circ)=& (3L-c_1)^3=27L^3-27c_1L^2+9c_1^2L-c_1^3,
\end{align*}
 which gives
 \begin{align*}
 3c_1^2L=& \Vol((3P)^\circ)-\Vol((2P)^\circ)-19\Vol(P)+15\F, \\
c_1^3=& 2\Vol((3P)^\circ)-3\Vol((2P)^\circ)-30\Vol(P)-42\F.
\end{align*}
The second formula follows from Proposition~\ref{prop:degree}.
The first formula can be deduced similarly.
 \end{proof}

We end this section with the (short!) classification of $2$-defective 
$2$-jet spanned  smooth toric
embeddings of dimension three.

\begin{cor}\label{cor:classification}
The only smooth, $2$-jet spanned toric embedding of a smooth threefold 
that is $2$-defective, is $(\pr{3},\mathcal O_{\pr{3}}(2))$.
\end{cor}

\begin{proof}
Assume $K_{X}+\eL$ is nef. Rewrite the formula in Theorem~\ref{thm:threefolds} as
\[5\Vol(P)+57(\Vol(P)-\F)+28\E-8\V+58 \Vol(P^\circ)+51\F_1+20\E_1.\]
Note that $\Vol(P)>0$, $\Vol(P)-\F=L^3-c_1L^2=L^2(L-c_1)\ge 0$, and that $\Vol(P^\circ), \F_1, \E_1 \ge 0$
 (since $L-c_1$ is nef).
As $P$ is $2$-regular, each edge of $P$ has length $\ge 2$ by Proposition~\ref{prop:DR99}. 
It follows that we can ``put'' a simplex at each vertex of $P$ 
such that the simplices don't overlap on the edges of $P$. 
This implies that $\E\ge 3\V$. Therefore, $28\E-8\V>0$.
Hence $c_3(\mathcal P^2_X(\eL))>0$ when  $K_{X}+\eL$ is nef, so $(X,\eL)$ is not $2$-defective.

If  $K_{X}+\eL$ is not nef, then by Lemma \ref{nef} there are only three cases to consider. If  $(X,\eL)=
(\pr{3},\mathcal O_{\pr{3}}(2))$, then it is easy to see that $j_2$ is an isomorphism, hence 
$\mathcal P^2_X(\eL)$ is trivial, and so  $c_3(\mathcal P^2_X(\eL))=0$. We have 
$\mathcal P^2_X(\ol_{\pr{3}}(3))=\oplus_1^{10}\ol_{\pr{3}}(1)$ and thus $c_3(\mathcal P^2_X(\mathcal L))=120$. 
%The map $\gamma_2$ has degree $1$ since the embedding is $3$-jet ample. 
Finally, the last case
$(X,\eL)=(\mathbb P(\mathcal O_{\mathbb P^1}(a),\mathcal O_{\mathbb P^1}(b), \mathcal O_{\mathbb P^1}(c)), 2\xi)$ 
was shown to be not $2$-defective in Example \ref{abc}.
\end{proof}

\section[A non 2-regular case]{Non k-regular toric varieties: 
the case of rational normal scrolls}\label{RNC}

 In this section we consider  rational normal scrolls. These varieties
 are toric (indeed they are defined by Cayley polytopes), but not 
 $k$-jet spanned for  $k \ge 2$.
%When a variety is not generically $k$-jet spanned, 
 Since the $k$th osculating spaces are defined using the image of the $k$th jet map,
 %, and hence also its $k$th dual variety. 
 when the rank of the $k$th jet map is strictly less than the rank of the $k$th jet bundle 
 $\mathcal P^k_X(\mathcal L)$, we cannot use the jet bundle to compute the degree of the 
 $k$th dual variety as  in the case of $k$-jet spanned varieties.
 In the case of scrolls, however, it is in some cases possible to identify 
 bundles that replace the jet bundles in the degree computations, see \cite{LMP08} and \cite{LMP11}. 
 Higher dual varieties of rational normal scrolls were studied in \cite{PS84}. 
 In the classical case $k=1$, the formula for the  dimension of the dual variety and a combinatorial 
 formula for its degree are  particular cases of results obtained by the study of  
 resultant varieties in~\cite[Section~5]{DFS07}.

Fix $n \ge 2$. 
Let $(X,\eL)=(\pr{}({\mathcal O_{\pr{1}}}(d_1)\oplus \ldots 
\oplus{\mathcal O_{\pr{1}}}(d_n)),\xi)$, where $0< d_1\leq \ldots 
\leq d_n$ and $\xi$ is the tautological line bundle of  
$\pr{}({\mathcal O_{\pr{1}}}(d_1)\oplus \ldots \oplus {\mathcal O_{\pr{1}}}(d_n)),$ 
be a smooth rational normal scroll. 
Let $m+1=\sum_1^n( d_i +1)$, so that the embedding is $X\hookrightarrow \pr{m}$.
Equivalently, $(X,\eL)$ is the toric variety associated to the Cayley 
polytope $P={\rm Cayley}(d_1\Delta_1,\ldots,d_n\Delta_1)$.

It is shown in \cite[Prop. 1, p. 1057]{PS84} that given $k$,  
$1\le k\le d_n$, by setting  $d_0=0$,  there is a uniquely defined integer $i_k$ such that
$$ 0\leq i_k\leq n-1\text{ and } d_{i_k}+1\leq k \leq d_{i_k+1},$$
and
\[\dim X^{(k)}=\left\{\begin{array}{ccc}
m+1-kn+\sum_{j=1}^{i_k} (k-1-d_j)&\text{ if }& 0\le i_k\leq n-2\\
m-kn+\sum_{j=1}^{n-1} (k-1-d_j)=d_n-k&\text{ if }& i_k=n-1
\end{array}\right.\]

\begin{prop}\label{degdual1} If $d_1\ge k$, then
$\dim X^{(k)}=m +1 -kn \text{ and } \deg_k X^{(k)}=kd-k(k-1)n$,
where $d=\sum_{i=1}^n d_i$ is the degree of $X$.
\end{prop}

\begin{proof} Since $d_1\ge k$ is equivalent to $i_k=0$, the assertion concerning the dimension follows
 from \cite[Prop. 1, p. 1057]{PS84}. In order to prove the formula for the degree, we shall use results 
from \cite{LMP08}. There, it is shown  that there exists a  bundle $\mathcal E_k^\vee$ representing the 
$k$th osculating spaces to $X$ almost everywhere. In the case of rational normal scrolls, if $k\le d_1$,  
the rank of the matrix $A_k$ of  \cite[p. 1050]{PS84}, which is equal to the rank of $j_k$, is equal to 
$kn+1$ everywhere, hence the bundle  $\mathcal E_k^\vee$ represents the  $k$th osculating spaces to $X$ 
everywhere and $\mathcal E_k^\vee={\rm Im} j_k$. Now $X^{(k)}$ is equal to the image of $\mathbb P({\rm Ker}
 j_k^\vee)\subseteq X\times {\mathbb P^m}^\vee$ via the second projection. Since the dimension of 
$\mathbb P({\rm Ker} j_k^\vee)$ is $m-kn-1+n$ and the dimension of $X^{(k)}$ is $m+1-kn$ (so that 
we may say that $X$ is $k$-defective with defect $m-kn-1+n-m-1+kn=n-2$ for $n\ge 3$), it follows that 
the degree of $X^{(k)}$ can be computed as follows:
\[ \deg_k X^{(k)} =p_*(\eta^{m+1-kn})\xi^{n-2}=s_2({\rm Ker} j_k^\vee)\xi^{n-2}=c_2(\mathcal E_k^\vee)\xi^{n-2},\]
where $p\colon \mathbb P({\rm Ker} j_k^\vee)\to X$ and 
$\eta$ denotes the tautological bundle on $\mathbb P({\rm Ker} j_k^\vee)$ 
 (the one coming from the tautological bundle on ${\mathbb P^m}^\vee$).
From \cite[p. 562]{LMP08} it follows that $c(\mathcal E_k^\vee)=(1+k(d-n(k-1))F)(1-2kF+\xi)$,
 where $F$ is the class of a fiber of $X\to \mathbb P^1$, so that
\[c_2(\mathcal E_k^\vee)\xi^{n-2}=kd-nk(k-1),
\]
since $F^2=0$ and $F\cdot \xi^{n-1}=1$.
\end{proof}

\begin{cor} 
Let $P={\rm Cayley}(d_1\Delta_1,\ldots,d_n\Delta_1)$ 
be a smooth Cayley polytope, $n\ge 2$, and $(X,\eL)$  
the corresponding toric embedding. 
For any integer $k$ with $1\le k\le \min\{d_1,\dots,d_n\}$, it holds that 
$\codim X^{(k)} = kn-1$ 
and $\deg_k   X^{(k)} = k\Vol(P)-\binom{k}2 \V$, where  $\V=2n$ denotes the number of vertices of $P$.
\end{cor}
Assume $k=1$. Since $i_1=0$ always, we have
\[ \dim X^{(1)}=m+1-n=m-1-(n-2),\]
and hence $X$ is defective with defect $n-2$ iff $n\ge 3$, as is well known. Moreover, 
for $k=1$, we get $\deg X^\vee =d =\deg X$, which is equally 
well known to hold for any ruled, non-developable variety.
\medskip

Consider the case $n=3$. The formula (3.6) of \cite[Thm. 3.4]{MT11} 
gives $\delta_1=0$ and $\delta_2=\deg X^\vee$, so that
\[\deg X^\vee = -2\Vol (P)+3\F-3\E+2\V,\]
which is easily computed to be equal to $d=d_1+d_2+d_3=\Vol(P)$, 
in agreement with Proposition \ref{degdual1}. For $k=2$, the 
formula of Proposition~\ref{degdual1} gives $\deg X^{(2)}=2(d-3).$ 
From Examples 3 and 4 in \cite[pp.~1055--57]{PS84} we know that 
if $d_1=d_2=d_3=2$, then $X^{(2)}$ is equal to the \emph{strict} 
dual variety of $X$ and is itself a rational normal scroll 
of the same type as $X$, in particular it has the same degree $d=6$. 
In the case $d_1=d_2=2, d_3=3$, then $X^{(2)}$ is a rational \emph{non-normal} 
scroll of type $(2,2,2,2)$; in particular it has degree $8$, 
which is in agreement with the formula of Proposition \ref{degdual1}, 
$\deg X^{(2)}=2(2+2+3-3)$.

\section{Tropicalization of higher duals of toric varieties}\label{def}

In this section, we consider equivariant embeddings of toric varieties of any dimension, 
not necessarily smooth and not necessarily normal. 
The aim is to describe the tropicalization 
of their $k$-th dual varieties, for any $k$.
Our results are a direct generalization of the corresponding results for the case
$k=1$ obtained in \cites{DFS07,DT11}. We refer the reader to these papers and to
the references therein (in particular, to Chapter 9 in~\cite{BSbook}),
 for the background of this section.

\subsection{Preliminaries on the tropicalization of algebraic
varieties}

Let $X \subset \pr{m}$ be a projective variety, $Y$ the (affine) cone over $X$ and $I$ their defining ideal.
For our purposes it will be enough to consider the case in which $I$ is defined over $\Q$, which we
view with the trivial valuation.
 Given a weight $w \in \R^{m+1}$ and a polynomial $F = \sum_{\alpha \in S} 
 F_\alpha x^\alpha \in \Q[x_0,\dots,x_m]$,
the initial polynomial ${\rm in}_w(F)$ is the subsum 
$\sum_{\langle \alpha, w \rangle =\mu} F_\alpha x^\alpha$ of terms where the minimum  
$\mu = {\rm min} \{ \langle \alpha, w \rangle, \, \alpha \in S, F_\alpha \neq 0\}$
is attained. 
The tropicalization of $Y$ then equals (as a set)
\begin{equation} \label{eq:tropY}
\trop(Y) \, = \, \{ w \in \R^{m+1}, \, \, {\rm in}_w(F) \text{\, is not a monomial for any\,} 
F \in I, f\neq 0\}.
\end{equation}
Thus, it coincides with those real weights $w$ for which the {\sl initial ideal
with respect to $w$} of the defining ideal $I$, contains no monomial.
Equivalently, given  an algebraically closed field $\mathbb{K}$ of characteristic
$0$ with a non-trivial non-Archimedean valuation ${\rm val}: \mathbb{K}^* \to \R$ and residue field
of characteristic zero, $\trop(Y)$ equals the closure  of  the (coordinatewise) image by ${\rm val}$ of 
the variety in the torus $(\mathbb{K}^*)^{m+1}$ defined by $I$, by Kapranov's Theorem \cite{EKL06}.

Let $T_m$ denote the torus of $\pr{m}$. Recall that  
$\trop(Y)$ is a rational polyhedral fan of the same dimension as the ``very affine'' variety $Y \cap T_m$
 \cite{BG84}, which captures the asymptotic
directions of $Y \cap T_m$ \cite[Prop.~2.3]{Tev}. 
We denote by $\R^{m+1}/{\sim}\,$ the quotient 
linear space,  where we identify a point $w \in \R^{m+1}$ with all points
in the line $L_w =\{w + \lambda (1, \dots, 1), \lambda \in \R\}$, and let
$\pi: \R^{m+1} \to \R^{m+1}/{\sim}\,$ be the projection. 
We can identify $\R^{m+1}/{\sim}$ with $\R^m$.
As the ideal $I$ is homogeneous, for any $w \in\trop(Y)$ it holds that
the whole line $L_w$ is contained in
$\trop(Y)$.  Thus, it makes sense to define the tropicalization
of $X$ as the projection $\trop(X) = \pi(\trop(Y))$.

\subsection{Parametrization of the higher dual varieties} %$X_\A^{(k)}$
 
Consider a configuration of lattice points 
$\A=\{{r}_0,\ldots,{r}_m\}\subset\Z^{n}$ as 
in Section~\ref{sec:stt}, and let $\iota_{\A}\colon X_\A \into \pr{m} $ 
be the associated toric embedding. Let 
$A\in \Z^{(n+1) \times (m+1)}$ denote the matrix 
with columns $\{(1, r_0), \ldots, (1, r_m)\}$ and let
 ${\mathbf 1}$ be the  point 
$(1:\cdots :1)$ in the torus of $X_\A,$ i.e., a general point of $X_\A.$ 
Then, $\pr{}({\rm Rowspan}(A))$  
can be identified with the embedded tangent space $ \T_{X_\A,{\mathbf 1}}$ to
$X_\A$ at this point.  

It is also straightforward to construct a matrix describing the higher osculating spaces
at $\mathbf 1$.
Given any matrix $A$ as above, call $\mb v_0 =(1, \ldots, 1), \mb v_1,
 \dots, \mb v_n \in \Z^{m+1}$ the row vectors of $A$.
Set $k=2$  and
denote by $\mb v_i*\mb v_j \in \Z^{m+1}$ the vector given by the coordinatewise 
product of these vectors. We define
the following matrix $A^{(2)} \in \Z^{\binom{n+2}2 \times (m+1)}$
{\tiny
\begin{equation} \label{eq:matrixA2}
A^{(2)} \, = \, \left(
\begin{array}{ccc} 
{} &\mb v_0 &{}\\
{}& \vdots& {}\\
{} & \mb v_n & {}\\
{} & \mb v_1 * \mb v_1 & {} \\
{} & \mb v_1 * \mb v_2 & {} \\
{} & \vdots & {} \\
{} & \mb v_{n-1} * \mb v_{n} \\
{} & \mb v_n * \mb v_n & {}
\end{array}\right),
\end{equation}
}
where the last rows are given by the products $\mb v_i*\mb v_j, \, 1 \le i \le j \le n$.
Then, $\pr{} ({\rm Rowspan}(A^{(2)}))$ $={\mathbb T^2_{X_\A,{\mathbf 1}}}$ describes the
second osculating space of $X_\A$ at the point ${\mathbf 1}$.

\begin{example}
Let $d \in \N$ and let $A \in \Z^{2\times d}$ be the matrix with rows $(1, \dots, 1), (0, 1, \dots,d)$.
The corresponding matrix $A^{(2)}$ is precisely the matrix $C$ in Example~\ref{ex:q=2}.
\end{example}

More generally, let $A$ be as above, and, for any $\alpha \in \N^{n+1}$, 
denote by $\mb v_\alpha= \alpha_0\mb v_0*\cdots *\alpha_n\mb v_n \in 
\Z^{m+1}$ the vector obtained as the coordinatewise product of  ($\alpha_0$ times the row
vector $\mb v_0$) times $\dots$ times ($\alpha_n$ times the row vector $\mb v_n$).  
Now, for given $k$, 
define the matrix $A^{(k)}$ as follows. Order the
vectors  $\{\mb v_\alpha : |\alpha| \le k\}$ (for instance, by degree and then by lexicographic
order with $0 > 1 > \dots > n$), and let  $A^{(k)}$ be the $\binom{n+k}k \times (m+1)$ 
integer matrix with these  rows. We then have:

\begin{lemma}\label{lemma:Tk}
The projectivization  of the rowspan of $A^{(k)}$ equals the $k$-th osculating space $
{\mathbb T^k_{X_\A,{\mathbf 1}}}$.  This space only depends on the toric variety $X_\A$ and not on
the choice of the matrix $A$ (and associated matrix $A^{(k)}$) we use to rationally parametrize the
variety $X_\A$. Moreover, $\iota_\A$ is generically $k$-jet spanned 
if and only if the rank of $A^{(k)}$ is maximal.
\end{lemma}

A configuration  $\A$ as above defines a torus action on $\pr{m}$ as follows:
$${\mb t} \cdot_A{\mb x}=({\mb t}^{{\mb r}_0}x_0:\cdots:{\mb t}^{{\mb r}_m}x_m).$$
Note that $X_\A=\overline{{\rm Orb}({\mb 1})}$ is the closure of the orbit of the
point ${\mb 1}$. For any $k$, the $k$-osculating spaces at the points in 
the torus of $X_\A$ are translated
by this action.  We deduce that the $k$-th dual variety equals
\begin{equation}\label{eq:orb}
X_\A^{(k)}=\overline{\bigcup_{{\bf y}\in {\rm Ker}(A^{(k)})}{\rm Orb}({\bf y})},
\end{equation}
where the orbits correspond to the action of the torus on the dual projective
space ${\mb t} \star'_A{\mb y}=({\mb t}^{-{\mb r}_0}y_0:\cdots:{\mb t}^{-{\mb r}_m}y_m).$
In other words, we can rationally parametrize $X_\A^{(k)}$ as follows. Denote by
$T_n$ the $n$-torus. The $k$-dual variety $X_\A^{(k)}$ coincides with the closure of
  the image of the map 
  \[\, \gamma'_k : \pr {} ({\rm Ker} (A^{(k)}))\times
  T_n \to (\pr {m})^\vee,\,\] given by
\begin{equation} \label{eq:phik}
 \gamma'_k(\mb y,\mb t) \,\,=\,\, {\mb t} \star_A{\mb y}. 
\end{equation}
That is, the $k$-th dual variety is equal to the closure of the orbits (of the
action with weights $\A$) of all the elements $y$ in the kernel of  $A^{(k)}$.

\subsection{Characterization of the tropicalizations of higher duals of projective toric varieties}
 
The explicit parametrization of higher duals of projective toric varieties allows us to give 
the following description of their tropicalized version $\trop(X_\A^{(k)}) \subset \R^m$
as the Minkowski sum of a classical linear space and a tropical linear space.
This result is a straightforward  generalization of Theorem 1.1 in \cite{DFS07}, but we
sketch the proof for the convenience of the reader.

\begin{thm}\label{th:trop}
The tropicalization $\trop(Y_\A^{(k)})\subset \R^{m+1}$ is equal to the Minkowski sum
$$\trop(Y_\A^{(k)})=\ {\rm Rowspan}(A)+ \trop({\rm Rowspan}(A^{(k)})).$$
Its image $\pi(\trop(Y_\A^{(k)}))$ in $\R^{m+1}/ \sim \,$ gives the tropicalization of the $k$-th dual
variety $X_\A^{(k)}$.
\end{thm}

\begin{proof}

As we can further parametrize ${\rm Ker} (A^{(k)})$ by linear forms, we can compose
the rational parametrization~\eqref{eq:phik} with a linear map to
get a rational parametrization of $X_\A^{(k)}$  
whose coordinates are given by monomials in linear forms. 
Note that as this parametrization  is defined over $\Q$,
the defining ideal of $X_\A^{(k)}$ is also defined over $\Q$.

We can compute its tropicalization by means of Theorem~3.1 in~\cite{DFS07}.
As the ideal of  ${\rm Ker} (A^{(k)})$ is just the space of 
linear forms vanishing on the rowspan of the matrix, and
the tropicalization of the torus action gives the linear space
${\rm Rowspan}(A)$,  we deduce the equality in the statement.
\end{proof}

\begin{remark}
In general, the intersection $X_\A^{(k)} \cap T_m$ of the $k$-th dual variety 
with the torus $T_m$ of $\pr{m}$ is dense in $X_\A^{(k)}$, 
but this might fail in ``border cases'' as the following example, in which
the tropicalization $\trop(X_\A^{(k)})$ is empty but the projective $k$-th dual variety is not, shows.
Let $\A \subset \Z^3$ be the configuration of lattice points given by  the point $r_0=(3,0,0)$
plus the ten lattice points in $2 \Delta_3$. This is a non smooth, generically $2$-jet spanned 
configuration. Its second dual variety $X_\A^{(2)}$ is non empty but it does not intersect $T_{10}$.   
Therefore, $\trop({\rm Rowspan}(A^{(k)}))$ is empty (and also $\trop(X_\A^{(k)})=\emptyset$). See Corollary~\ref{cor:empty} for a precise characterization of when this happens.
\end{remark}

\begin{remark}\label{rmk:extension}
We can also extend Corollary~4.5 and Theorem~4.6 in \cite{DFS07} to
compute tropically the dimension and the degree of higher dual varieties associated to
toric embeddings. We omit here precise statements and 
proofs, since they are direct generalizations of the results in \cite{DFS07}
and we would need to introduce several definitions. 

The dimension of $X_\A^{(k)}$ coincides
with the dimension of its tropicalization (provided $X_\A^{(k)}\cap T_m\neq \emptyset$), 
which by Theorem~\ref{th:trop} equals the maximum of the dimensions
of the sum of a cone in the tropical linear space $\trop({\rm Rowspan}(A^{(k)}))$
and the linear space ${\rm Rowspan}(A)$. 
Thus, $\dim X_\A^{(k)}$ can be computed via a description of the cones in
 $\trop({\rm Rowspan}(A^{(k)})$. This tropical linear space 
is also denoted by
${\mathcal B}^*(A^{(k)})$ and called the co-Bergman fan of $A^{(k)}$ \cite{BSbook}.
This space is well
studied, and a characterization which allows for an efficient implementation called
{\sl TropLi}, can be found in \cite{R}. 

On the other hand, the tropical computation of $\deg_k(X_\A^{(k)})$
can be carried out as in \cite[Theorem~2.2]{DFS07}. Given a generic weight $w \in \R^{m+1}$,
consider the initial ideal of the vanishing ideal of $X_\A^{(k)}$ with respect to $w$. 
The multiplicities
of the monomial primary components of this initial ideal, can be translated into
tropical multiplicities and characterized as the
sum of the absolute value of certain minors of $A$, which can be explicitly
described in terms of $w$ and chains of the supports of the vectors in the kernel of $A^{(k)}$.
However, we do not know how to make a  geometrical interpretation of
the terms in this sum. 
For classical discriminants, the algorithm deduced from Theorem~4.6 in \cite{DFS07} 
has also been implemented by F. Rinc\'on \cite{R}, and could be extended to deal
with degree computations for higher duals.
\end{remark}

We are now interested in a direct
description of when a given weight lies in this tropicalization, in terms
of its induced \emph{coherent marked subdivision}
 of the convex hull $N(A)$ of $\A$ \cite{GKZ}. 
A weight vector
 $w\in \R^{m+1}$  defines the  coherent marked
subdivision of  $N(A)$ given by the 
collection of subsets of $A$ corresponding to
the domains of linearity of the collection of affine functions
describing the faces of the lower convex hull of the set of lifted points $\{(r_i,
w_i), i=0,\dots, m\}$ in $\R^{n+1}$. 

Note that $X_\A^{(k)}$ can be identified with the closure in the parameter space 
(with variables $x = (x_0, \dots, x_m)$) of the vectors of coefficients
$x$ of polynomials $F_\A(x,t) =\sum_{i=0}^m x_i t^{r_i}$ in $n$ variables with support $\A$, 
for which the hypersurface of the $n$-torus $\{t : \ F_\A(x,t) =0\}$
has a singular point, that is,  where $F_\A$ and all its partials 
(with respect to the $t$-variables) up to order $k$ vanish.
This point of view leads to a characterization of
the points $u \in \trop(Y_\A^{(k)})$ via a generalization
of Theorem~2.9 in \cite{DT11}.

We denote by $\oplus, \odot$ 
the tropical operations in $\R \cup \{+\infty\}$
given by $a \oplus b:= min\{a,b\}$ and $a \odot b:= a+b$.
Given $u \in \R^{m+1}$, consider the tropical polynomial on $\R^{m+1}$ defined by
$$p_{\A,u}(w) = \oplus_{i=0}^m u_i \odot w^{r_i},$$ 
where $w^{r_i}$ is understood
tropically. Thus, $ p_{\A,u}(w)$ equals the minimum of the linear forms 
$u_i +\langle w, r_i \rangle$ for $i=0, \dots, m$.
For  any (not necessarily homogeneous) polynomial $Q(y_0, \dots, y_n)$ in
$\R[y_0, \dots, y_n]$, we define the Euler derivative $\frac{\partial p_{\A,u}}{\partial Q}$ with respect to
$Q$, as the subsum of those terms in $p_{\A,u}$  corresponding to all points $r_i \in \A$ for
which $Q(r_i) \not=0$. In particular, when $Q$ is the constant polynomial $1$, the corresponding
Euler derivative equals $p_{\A,u}$. 

We consider tropical polynomials $p_{\A,u}(w)$ with vector
of coefficients  $u ={\rm val}(x)$ (that is, $u=({\rm val}(x_0), \dots, {\rm val}(x_m))$)
 and we translate tropically the conditions of the
 vanishing of all partials of the polynomial $F_\A(x,t)$ at some point $t$ in the torus.
We have the following theorem.
%As usual, given $k \in \N$, 
%we will denote by $\R[y_0, \dots, y_n]_{\le k}$ the subspace of polynomials of degree at most $k$.

\begin{thm}\label{thm:euler}
Let $\A$ be a lattice configuration as above, and fix $k \in \N$.
A point $u \in \R^{m+1}$ lies in the tropicalization $\trop(Y_\A^{(k)})$ of the cone
over the $k$-th dual variety of the associated toric variety $X_\A$
if and only if 
\begin{equation}\label{eq:Q}
\bigcap_{Q \in \Q[y_1, \dots, y_n], \deg(Q) \le k} V(\frac{\partial p_{\A,u}}{\partial Q})\neq \emptyset.
\end{equation} 
This intersection is given by a finite number of Euler derivatives of $p_{\A,u}$.
\end{thm}

The case $k=1$ is precisely  the content of Theorem~2.9 in~\cite{DT11}. 
The proof of the general case follows the same lines.  Therefore  we only give a sketch
here and refer
to that proof (and the previous results in~\cite{DT11}) for the details. 

\begin{proof}
Call $W_u$ the set defined by the intersection in \eqref{eq:Q}.
For each subset $A$ of $\A$ for which there exists a rational
polynomial $Q_A$ of degree at most $k$ whose zero locus intersects $\A$ 
on $A$, pick such a polynomial $Q_A$. Therefore, $W_u$ equals 
the finite intersection corresponding to these polynomials $Q_A$ and
it is a closed set. 

As before, let $\mathbb{K}$ 
be an algebraically closed field of characteristic
$0$ with a non-trivial non-Archimedean valuation ${\rm val}\colon \mathbb{K}^* \to \R$,
with residue field of characteristic zero. We may moreover assume that the image
of the valuation is $\R$.
Assume $u$ lies in the tropicalization
of $Y_\A^{(k)}$, with 
$u ={\rm val}(x) =({\rm val}(x_0), \dots, {\rm val}(x_m))$, with $x \in ({\mathbb K}^*)^{m+1}$,
%Let $F_\A(x,t) \in {\mathbb K}[t_0, \dots, t_m]$ with support $\A$,
and there exists a point $q \in (\K^*)^{n}$ for which $F_\A(x,t)$ and all its partials 
(with respect to the $t$-variables) up to order $k$ vanish. 
 Note that 
for any rational polynomial $Q = \sum_{|\alpha| \le k} Q_\alpha t^\alpha$ of degree at most
$k$, the polynomial $E_Q(F)(x,t):=\sum_{i=0}^m Q(r_i) x_i t^{r_i}$ also vanishes at $q$.
Then, the vector $b = {\rm val}(q)$ lies in the tropical zero set of the tropicalization
of $E_Q(F)$, which equals $\frac{\partial p_{\A,u}}{\partial Q}$.
Therefore, the vector $b$ lies in $V(\frac{\partial p_{\A,u}}{\partial Q})$, and so
the intersection  \eqref{eq:Q} is non empty, proving the ``only if'' statement.

To prove the converse,  let $b$ be a point in $W_u$.
Then, for any rational polynomial $Q$ of degree at most $k$, 
the minimum of the linear forms $u_i +\langle b, r_i \rangle$ is attained for at least
two different indices $i_1, i_2$ for which $Q(r_{i_1}) \not=0, Q(r_{i_2}) \not=0$.
This happens if and only if for all these $Q$, 
the point $(u, b)$ lies in  $V(L_Q)$, where $L_Q$ is the tropical
polynomial in $(m+1) +n$ variables defined by 
$L_Q(v,w):= \bigoplus_{r_i \in A-\{Q=0\}} v_i \odot w^{r_i}$.
Note that any vector in the rowspan of $A^{(k)}$ is of the form 
$(Q(r_0), \dots, Q(r_m)), $
where $Q = \sum_{\alpha \le k} Q_\alpha t^\alpha$ is a polynomial of degree at most $k$. 
With the same arguments as in Proposition~2.8 in \cite{DT11}, we see
that these polynomials $L_Q$ form a tropical basis of the incidence
variety ${\mathcal H}_k$ of those $(x, q)\in \K^{(m+1) \times n}$ for which $q$ is a singular
point of $F_\A(x,t)$ where all derivatives up to order $k$ vanish.
By Kapranov's theorem, there is a point
$(x,b)\in {\mathcal H}_k$ such that $u = {\rm val}(x)$ and $b = {\rm val} (q)$, and so
$u \in \trop(Y_\A^{(k)})$, as wanted.
\end{proof}

We also deduce the following.

\begin{cor}\label{cor:empty}
Let $\A$ be a lattice configuration as above and  $k \in \N$. Then, $\trop(X_\A^{(k)})$ is empty (or
equivalently, $X_\A^{(k)}$ does not intersect the torus $T_m$ of $\pr{m}$) if and only if there exists a  
polynomial $Q$ of degree at most $k$ which vanishes at all  points in $\A$ but one. 
\end{cor}

\begin{proof}
It follows from~\eqref{eq:orb} that  $X_\A^{(k)} \cap T_m = \emptyset$ if and only if
${\rm Ker} (A^{(k)}) \cap T_m = \emptyset$. This is equivalent to the existence of a linear
form with support a single variable, let's say $x_m$, in the ideal of this kernel; that is, to the
fact that the vector $(0,0, \dots,1)$ lies in the rowspan of $A^{(k)}$.
But, as we remarked in the proof of Theorem~\ref{thm:euler},
the linear forms in the rowspan of $A^{(k)}$ are exactly those 
with coefficients $(Q(r_0), \dots, Q(r_m))$, where $Q$ 
runs over all polynomials $Q$ of degree at most $k$.
So, $X_\A^{(k)} \cap T_m = \emptyset$ if and only if 
there exists a polynomial $Q$ of degree at most
$k$ vanishing at all points of $A$ except at $r_m$.
\end{proof}

Let us note that already in the case $k=1$, 
it is not enough in general to consider in \eqref{eq:Q}  the vanishing
of only $\binom{n+k}k$ derivatives (see \cite[Example~2.5]{DT11}).

We end with a simple example
where we show a tropical  curve corresponding to coefficients $u$ such
that $\pi(u)  \in \trop(X_\A^{(2)})$.

\begin{example}\label{ex:trop}
Let $\A = 3 \Delta_2 \cap \Z^2$,
where  we order the points as follows : 
\[\A \, = \{ (0,0), (1,0), (2,0), (3,0), (0,1), (1,1), (2,1), (0,2), (1,2),(0,3)\},\]
and we write the generic polynomial with support in $\A$
\[f_\A(x,t) \, = \, x_0 \, t^{(0,0)} +\dots + x_9 \,  t^{(0,3)}, \quad t = (t_1, t_2).\] 
The linear space 
${\rm Ker}_{\mathbb K}(A^{(2)})$ consists
of all vectors $x \in {\mathbb K^*}^{10}$ of coefficients of polynomials $f_\A(x,t)$ with derivatives
vanishing up to order $2$ at the point $\mathbf 1$. 
When all $x_i \not=0$, the real vector $u = {\rm val}(x) \in \R^{10}$ 
lies in $\trop(Y_\A^{(k)})$. Moreover, the
tropical curve $V( p_{\A,u})$ has a singular point at $(0,0) ={\rm val}(\mathbf 1)$, and the origin is in 
the intersection of the loci  $V(\frac{\partial p_{\A,u}}{\partial Q})$ of 
all Euler derivatives corresponding to polynomials $Q$ of degree at most two.
This means that if we delete all points lying on any conic, there is a tie in the
minimum of the valuations of at least two of the remaining points.
A choice is given by the vector 
$u =(4,1,2,3,1,1,2,1,1,1)$.
In Figure~1, we  show $V(p_{\A,u})$ on the left. The three rays meeting at the rightmost vertex
have multiplicity one, and the others have multiplicity two. On the right we also depict
$V(\frac{\partial p_{\A,u}}{\partial Q_j})$ for $Q_1(w_1, w_2)=(w_1+w_2-1)(w_1+w_2-2)$,
$Q_2(w_1, w_2) = w_1 w_2$.
We see that the origin is the only point in $V(p_{\A,u})$ which lies in the intersection
of these tropical varieties and so $(0,0)$ is the only singular point of
$V(p_{\A,u})$.
\begin{figure}
\begin{center}
 {}{}\scalebox{0.40}{\includegraphics{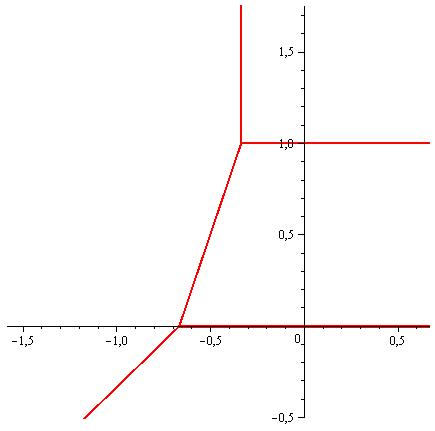}}
  {}{}\scalebox{0.40}{\includegraphics{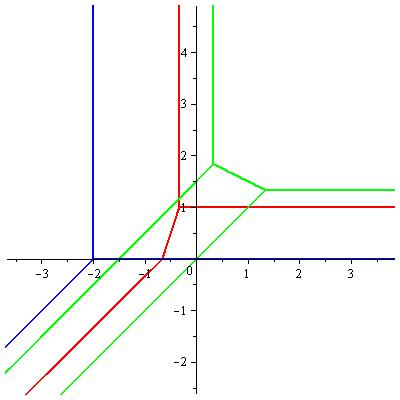}}
 \label{fig:1}
 \caption{The tropical curve $V(p_{\A,u})$ and the curves of two of its second Euler derivatives. 
(The figures were made using The Tropical Maple Package \cite{Trop}.)}
 \end{center}
 \end{figure}
The second dual variety $X_\A^{(2)} \subset (\mathbb{P}^{9})^\vee$ has the expected dimension 
$2 + 9 -5 =6$. In this small case, we can compute its ideal $I$ using any Computer
Algebra System. Using Singular \cite{Sing}, we found for instance the following
polynomial $h \in I$:
\[ h \, = \,  4 x_6 x_7^2 - 4 x_5 x_7 x_8 + 4 x_4 x_8^2 + 3 x_5^2 x_9 - 12 x_4 x_6 x_9.
\]
The $u$-weights of the five monomials in $h$ are respectively equal to $(2 +2, 1+1+1, 1+2,2+1, 1+1+1)$.
We thus check that ${\rm in}_u(h) =  - 4 x_5 x_7 x_8 + 4 x_4 x_8^2 + 3 x_5^2 x_9$ is not a monomial,
as predicted. However, note that in cases where the ideal cannot be computed, we can still determine
(via Theorems~\ref{th:trop} and~\ref{thm:euler}) all weights in the tropicalization.
% another polynomial in $I$
%$ x_4 x_5 x_7 - 2 x_1 x_7^2 - 2 x_4^2 x_8 + 6 x_0 x_7 x_8 + 6 x_1 x_4 x_9 - 9 x_0 x_5 x_9$
%$(1+1+1, 1+2, 2+1, 4+1+1, 1+1+1, 4+1+1)$
\end{example}

\noindent {\bf Acknowledgements:} We are grateful to the Royal Institute 
of Technology and to the Centre of Mathematics for Applications 
at the University of Oslo,  which hosted several of
our first discussions, and to the Institut Mittag-Leffler, 
where the main part of this work was completed. We thank the G\"oran Gustafsson Foundation
for providing financial support for our stays in Stockholm. We thank the referee
for a careful reading of our manuscript and for pointing out several typos in the formulas. We also thank Christian J\"urgens for the correction of a statement in our previous version.

\raggedright
\end{document}